\documentclass[12pt]{amsart}
\usepackage[T2A]{fontenc}
\usepackage[cp1251]{inputenc}
\usepackage[english]{babel}
\usepackage{amsmath,latexsym,amsthm,amsfonts,tipa,upgreek}
\usepackage{amssymb,amscd,graphpap, stmaryrd}

\newcommand\NN{{\mathbb N}}

\newcommand\w{{\omega}}
\newcommand\kk{{\varkappa}}

\newcommand\PP{{\mathcal P}}

\newtheorem{Th}{Theorem}

\newtheorem{Qs}{Question}

\theoremstyle{definition}

\begin{document}

\title{A Note on Partitions of Groups}
\author{Igor Protasov, Sergii Slobodianiuk}
\date{}
\address{Department of Cybernetics, Kyiv University, Volodymyrska 64, 01033, Kyiv, Ukraine}
\email{i.v.protasov@gmail.com; }
\address{Department of Mechanics and Mathematics, Kyiv University, Volodymyrska 64, Kyiv, Ukraine}
\email{slobodianiuk@yandex.ru}
\maketitle

\begin{abstract} Every infinite group $G$ of regular cardinality can be partitioned $G=A_1\cup A_2$ so that $G\neq FA_1$, $G\neq FA_2$ for every subset $F\subset G$ of cardinality $|F|<|G|$.
In \cite[Problem 13.45]{b3}, the first author asked whether the same is true for each group $G$ of singular cardinality.
We show that an answer depends on the algebraic structure of $G$.
In particular, this is so for each free group but the statement does not hold for every Abelian group $G$ of singular cardinality.
As an application, we prove that every Abelian group of singular cardinality $\kk$ admits maximal translation invariant $\kk$-bounded topology that impossible for all groups of regular cardinality. \end{abstract}

\section{Introduction}

Let $G$ be a group and let $\kk$ be a cardinal, $\kk\le|G|$. 
We denote by $[G]^{<\kk}$ the family of all subsets of $G$ of cardinality $<\kk$, and by $cf\kk$ the cofinality of $\kk$, so an infinite cardinal $\kk$ is regular if and only if $cf\kk=\kk$.
We say that a subset $A$ of $G$ is
\begin{itemize}
\item{} {\em left (right) $\kk$-thick} if, for every $F\in [G]^{<\kk}$ there exists $a\in A$ such that $Fa\subseteq A$ ($aF\subseteq A$);
\item{} {\em $\kk$-thick} if, for every $F\in [G]^{<\kk}$ there exists $a\in A$ such that $FaF\subseteq A$;
\item{} {\em left (right) $\kk$-large} if, there exists $F\in [G]^{<\kk}$ such that $G=FA$ ($G=AF$);
\item{} {\em $\kk$-large} if, there exists $F\in [G]^{<\kk}$ such that $G=FAF$.
\end{itemize}

We observe that $A$ is left $\kk$-thick (right $\kk$-thick, $\kk$-thick) if and only if $G\setminus A$ is not left $\kk$-large (right $\kk$--large, $\kk$-large),
$A$ is left $\kk$-thick (left $\kk$-large) if and only if $A^{-1}$ is right $\kk$-thick (right $\kk$-large), and $A$ is $\kk$-thick ($\kk$-large) if and only if $A^{-1}$ is $\kk$-thick ($\kk$-large).

If $A$ is $\kk$-thick then $A$ is left and right $\kk$-thick. If $A$ is either left or right $\kk$-large then $A$ is $\kk$-large. 
We show that the conversions of these statements do not hold.

For a free group $F_A$ in the alphabet $A$ and $g\in F_A$, $g\neq e$, we denote by $\lambda(g)$ and $\rho(g)$ the first and the last letter in the canonical representation of $g$,
$alph(g)$ denotes the set of all $a\in A$ such that either $a$ or $a^{-1}$ occurs in the canonical representation of $g$.

Now let $|A|>1$, $\kk=|F_A|$. We fix $a\in A$ and put $$S=\{g\in F_A:\lambda(g)\in\{a,a^{-1}\},\rho(g)\in\{a,a^{-1}\}\},$$
$K=\{e,a,a^{-1}\}$. Then $KSK=F_A$ so $S$ is $\kk$-large. Suppose that $G=HS$ for some $H\in[F_A]^{<\w}$.
If $|A|\le\aleph_0$ then $H$ is finite. We take $b\in A\setminus\{a\}$ and choose $n\in\NN$ so that $b^n\notin HS$.
If $|A|>\aleph_0$, we take $c\in A\setminus\{a\}$ such that $c\notin alph(h)$ for each $h\in H$. Then $c\notin HS$.
In both cases, $S$ is not left $\kk$-large. Since $S=S^{-1}$, we see that $S$ is not right $\kk$-large.
Hence, $G\setminus S$ is left and right $\kk$-thick but $G\setminus S$ is not $\kk$-thick.

Let $G$ be an infinite group and let $\kk$ be an infinite cardinal, $\kk\le|G|$.
We say that $G$ is {\em $\kk$-normal} if every subset $F\in[G]^{<\kk}$ is contained in some normal subgroup $H\in[G]^{<\kk}$.
If $G$ is $\kk$-normal then three "thick" sizes and three "large" sizes of subsets of $G$ coincide.
Clearly, each infinite Abelian group $G$ is $\kk$-normal for each $\kk\le|G|$.

\section{Results}~\label{s2}

Our first result follows from the ballean Theorem 3.3 in \cite{b9}. For topological prehistory of this theorem see \cite{b2}, \cite{b4}.
\begin{Th}\label{t1} Every infinite group $G$ can be partitioned into $|G|$ $\kk$-thick subsets provided that either $\kk$ is infinite and $\kk<|G|$ or $\kk=|G|$ and $\kk$ is regular. \end{Th}

Surprisingly, the singular case is "cardinally" different.
\begin{Th}\label{t2} For every group $G$ of singular cardinality $\kk$, the following statements hold
\begin{itemize}
\item[{\it (i)}] if a subset $A$ of $G$ is left $\kk$-thick then $A$ is right $\kk$-large;
\item[{\it (ii)}] $G$ cannot be partitioned into two $\kk$-thick subsets;
\item[{\it (iii)}] if $G$ is $\kk$-normal then, for every finite partition $G=A_1\cup\dots\cup A_n$, at least one cell $A_i$ is $\kk$-large.
\end{itemize}
\end{Th}
\begin{proof} $(i)$ We write $G$ as a union $G=\bigcup\{H_\alpha:\alpha<cf\kk\}$ of subsets from $[G]^{<\kk}$.
For each $\alpha<cf\kk$, we pick $g_\alpha\in A$ such that $H_\alpha g_\alpha\subseteq A$ so $H_\alpha\subseteq Ag^{-1}_\alpha$.
We put $F=\{g_\alpha^{-1}:\alpha<cf\kk\}$. Then $F\in[G]^{<\kk}$ and $G=\bigcup\{H_\alpha:\alpha< cf\kk\}\subseteq AF$.
Hence, $A$ is right $\kk$-large.

$(ii)$ Suppose that $G$ is partitioned $G=A\cup B$ such that $A$ is $\kk$-thick.
Then $A$ is left $\kk$-thick and, by $(i)$ $A$ is right $\kk$-large. 
Hence, $G\setminus A$ is not right $\kk$-thick and $B$ is not $\kk$-thick.

$(iii)$ We proceed by induction. For $n=1$ the statement is evident. Let $G=A_1\cup\dots\cup A_{n+1}$.
We put $B=A_2\cup\dots\cup A_{n+1}$. If $A_1$ is not large then $B$ is thick.
By $(i)$, $B$ is large so $G=FB$ for some $B\in[G]^{<\kk}$. 
Since $G=FA_2\cup\dots\cup FA_{n+1}$, by the inductive hypothesis, there exists $i\in\{2\dots n+1\}$ such that $FA_i$ is large.
Hence $A_i$ is large.
\end{proof}

By Theorem~\ref{t1}, every infinite group $G$ of regular cardinality $\kk$ can be partitioned $G=B_1\cup B_2$ so that each subset $B_i$ is not left $\kk$-large.
In \cite[Prolem 13.45]{b3}, the first author asked if the same is true for every group $G$ of singular cardinlity $\kk$.
Theorem~\ref{t2}$(iii)$ gives a negative answer for every $\kk$-normal group of singular cardinality $\kk$.
On the other hand, this is so for every free group.
\begin{Th}\label{t3} Every free group $F_A$ can be partitioned $F_A= B_1\cup B_2$ so that each cell $B_i$ is not left $\kk$-large. \end{Th}
\begin{proof} In view of Theorem~\ref{t1}, we may suppose that $|A|>\aleph_0$, so $|A|=|F_A|$. We partition $A=A_1\cup A_2$ so that $|A|=|A_1|=|A_2|$, and put
$$B_1=\{g\in F_A:\rho(g)\in A_1\cup A_1^{-1}\},\text{ } B_2=F_A\setminus B_1.$$
Assume that $F_A=HB_1$ for some $H\in[F_A]^{<\kk}$, where $\kk=|F_A|$.
Then we choose $c\in A_2$ such that $c\notin alph(h)$ for any $h\in H$.
Clearly, $c\notin HB_1$ and $B_1$ is not left $\kk$-large. Analogously, $B_2$ is not left $\kk$-large.
\end{proof}

\section{Applications}

A topological space $X$ with no isolated points is called {\em maximal} if $X$ has an isolated point in any stronger topology.
We note that $X$ is maximal if and only if, for every point $x\in X$, there is only one free ultrafilter converging to $x$.

A topology $\tau$ on a group $G$ is called {\em left invariant} if, for every $g\in G$, the left shift $x\mapsto gx:\text{ }G\to G$ is continuous in $\tau$.
A group $G$ endowed with a left invariant topology is called {\em left topological}. 
We say that a left topological group $G$ is maximal if $G$ is maximal as a topological space.
By \cite[$\S$2]{b5}, every infinite group $G$ of cardinality $\kk$ admits $2^{2^\kk}$ distinct maximal left invariant topologies.

A left topological group $G$ of cardinality $\kk$ is called $\kk$-bounded if each neighborhood $U$ of the identity $e$ (equivalently, each non-empty open subset of $G$) is left $\kk$-large.
We note that each left $\kk$-thick subset of $G$ meets every left $\kk$-large subset of $G$. 
It follows that $A$ is dense in every $\kk$-bounded topology on $G$.
Hence, if $G$ can be partitioned into two left $\kk$-thick subsets then every $\kk$-bounded topology on $G$ is not maximal.
Applying Theorems ~\ref{t1} and ~\ref{t3}, we get the following theorem.

\begin{Th} An infinite group $G$ of cardinality $\kk$ admits no maximal left invariant $\kk$-bounded topology provided that either $\kk$ is regular or $G$ is a free group. \end{Th}

The following theorem answers affirmatively Question 4.4 from \cite{b5}.
\begin{Th}\label{t5} Every $\kk$-normal group $G$ of singular cardinality $\kk$ admits a maximal left invariant $\kk$-bounded topology. \end{Th}
\begin{proof}
We use a technique from \cite{b1}: endow $G$ with the discrete topology, identify the Stone-$\check{C}$ech compactification 
$\beta G$ of $G$ with the set of all ultrafilters on $G$ and consider $\beta G$ as a right topological semigroup. We put
$$L=\{q\in\beta G:\text{ each member } Q\in q\text{ is left }\kk\text{-large}\}.$$
Applying Theorem~\ref{t2}$(iii)$ and Theorem 5.7 from \cite{b1}, we conclude that $L\neq\varnothing$. 
It is easy to see that $L$ is a closed subsemigtoup of $\beta G$ so $L$ has some idempotent $p$.
The ultrafilter $p$ defines a desired topology on $G$ in the following way:
for each $g\in G$, the family $\{gP\cup\{g\}:P\in p\}$ forms a base of neighborhoods of $g$.
\end{proof}

We note that the set $L$ from the proof of Theorem~\ref{t5} is a left ideal in $\beta G$.
Hence, the idempotent $p$ could be chosen from some minimal left ideal of $\beta G$.
In this case, each member $P$ of $p$ has a rich combinatorial structure (see \cite[Chapter 14]{b1}).

\section{Comments and Open Questions}

$1.$ For a partition $\PP$ of $G$, we denote by $\PP\wedge\PP^{-1}$ the partition of $G$ into the cells $\{A\cap B^{-1}: A,B\in\PP\}$.
If every cell of $\PP$ is not left $\kk$-large then each cell of $\PP\wedge\PP^{-1}$ is neither left nor right $\kk$-large.
On the other hand, if $|A|$ is singular and $F_A=B_1\cup B_2$, by Theorem~\ref{t2}$(i)$, at least one cell $B_i$ is either left or right $\kk$-large.

We show that a free group $F_A$ of cardinality $\kk$ can be partitioned $F_A=B_1\cup B_2\cup B_3$ so that each subset $B_i$ is neither left nor right $\kk$-large.

If $A$ is infinite, we partition $A$ into three subsets $A=A_1\cup A_2\cup A_3$ of cardinality $\kk$ and put
$$B_1=\{g\in F_A:\lambda(g)\in(A_2\cup A_3)^{\pm1},\text{ }\rho(g)\in(A_2\cup A_3)^{\pm1}\},$$
$$B_2=\{g\in F_A:\lambda(g)\in(A_1\cup A_3)^{\pm1},\text{ }\rho(g)\in(A_1\cup A_3)^{\pm1}\}\setminus B_1,$$
$$B_3=F_A\setminus(B_1\cup B_2).$$

If $2<|A|<\aleph_0$, we partition $A$ into $3$ non-empty subsets $A_1$, $A_2$, $A_3$ and define $B_1,B_2,B_3$ as above.

If $|A|=1$, one can easily partition $F_A$ even into two non $\aleph_0$-large subsets.

At last, let $A=\{a,b\}$. For each word $y\in F_A$ of length $\ge2$, we denote by $\lambda_2(g)$ and $\rho_2(g)$ the first and the last factor of $g$ of length $2$. We put
$$B_1=\{g\in F_A:\lambda_2(g),\rho_2(g)\in \{a^{\pm2},a^{\pm1}b^{\pm1},b^{\pm1}a^{\pm1}\}\},$$
$$B_2=\{g\in F_A:\lambda_2(g),\rho_2(g)\in \{b^{\pm2},a^{\pm1}b^{\pm1},b^{\pm1}a^{\pm1}\}\}\setminus B_1,$$
$$B_3=F_A\setminus(B_1\cup B_2).$$

$2.$ Let $G$ be an infinite group and let $\kk$ be an infinite cardinal, $\kk\le|G|$.
For each $x\in G$, we denote $x^G=\{g^{-1}xg:g\in G\}$.
If $G$ is $\kk$-normal then $|x^G|<\kk$ for each $x\in G$.

We show that Theorem~\ref{t2}$(iii)$ fails to be true if we replace "$G$ is $\kk$-normal" to "$|x^G|<\kk$ for every $x\in G$".

Let $\kk$ be a singular cardinal, $\kk=\bigcup_{\alpha<cf\kk}\kk_\alpha$, $\kk_\alpha<\kk$.
For each $\alpha<cf\kk$ we take an alphabet $A_\alpha$ such that the subsets $\{A_\alpha:\alpha<\kk\}$ are pairwise disjoint and $A_\alpha=\kk_\alpha$.
We consider the direct sum $G=\oplus_{\alpha<cf\kk}F_{A_\alpha}$, denote by $\pi_\alpha: G\to F_{A_\alpha}$ the canonical projection, by $e_\alpha$ the identity of $F_{A_\alpha}$ and fix some $a_\alpha\in A_\alpha$.
For each $g\in G\setminus\{e\}$, we take $\alpha<cf\kk$ such that $\pi_\alpha(g)\neq e_\alpha$ but $\pi_beta(g)=e_\beta$ for each $\beta>\alpha$.
We denote by $\rho(g)$ the last letter of $\pi_\alpha(g)$ and put $B=\{g\in G\setminus\{e\}:\rho(g)\in\{a_\alpha^{\pm1}\},\alpha<cf\kk\}$.
Then $B$ and $G\setminus B$ are not $\kk$-large but $|x^G|<\kk$ for each $x\in G$

$3.$ In \cite{b11} we conjectured that every infinite group $G$ of cardinality $\kk$ can be partitioned 
$G=\bigcup_{n\in\w}A_n$ so that each $A_nA_n^{-1}$ is not left $\kk$-large (and so is not right $\kk$-large).
We confirmed this conjecture for every group of regular cardinality and for some (in particular, Abelian and free) groups of an arbitrary cardinality.
In the general case this problem remains open. Now we formulate a weaker version of the conjecture.
\begin{Qs}\label{q2} Can every infinite group $G$ of cardinality $\kk$ be partitioned $G=\bigcup_{n\in\w}A_n$ so that each $A_n$ is not left $\kk$-large. \end{Qs}

A subset $A$ of a group $G$ is called
\begin{itemize}
\item{\em left (right) $\kk$-small} if $L\setminus A$ is left (right) $\kk$-large for each left (right) $\kk$-large subset $L$ of $G$;
\item{\em $\kk$-small} if $A$ is left and right $\kk$-small.
\end{itemize}

By \cite{b7}, every infinite group $G$ can be partitioned into $\aleph_0$ small subsets.
By \cite[Theorem 5.1]{b10}, $G$ can be partitioned in $\aleph_0$ $\kk$-small subsets for each cardinal $\kk$ such that $\aleph_0\le\kk\le cf|G|$. 
The next question from \cite[Problem 4.2 and 4.3]{b8}
\begin{Qs}\label{q3} Can every infinite group $G$ be partitioned into $\aleph_0$ $|G|$-small subsets? \end{Qs}

We note that an affirmative answer to Question~\ref{q3} would imply an affirmative answer to Question~\ref{q2}.

$4.$ By \cite{b6}, every infinite group $G$ can be partitioned into $\aleph_0$ left and right $\aleph_0$-large subsets.
By \cite[Theorem 2.4]{b8}, every infinite group $G$ can be partitioned into $\kk$ left $\kk$-large subsets for each $\kk\le|G|$.
The next two questions from \cite{b8}.

\begin{Qs} Let $G$ be an infinite group and let $\kk$ be an infinite cardinal $\le|G|$. 
Can $G$ be $\kk$-partitioned so that each cell fo the partition is left and right $\kk$-large?\end{Qs}

Let $G$ be an infinite amenable (in particular Abelian) group and let $\mu$ be a left invariant Banach measure on $G$.
Clearly $\mu(A)>0$ for every left $\aleph_0$-large subset $A$ of $G$.
It follows that $G$ cannot be partitioned into $\aleph_1$ left $\aleph_0$-large subsets.
On the other hand, every free group of infinite rank $\kk$ can be easily partitioned into $\kk$ left $\aleph_0$-large subsets.

\begin{Qs} Let $G$ be a free Abelian group of rank $\aleph_2$. 
Can $G$ be partitioned into $\aleph_2$ $\aleph_1$-large subsets? \end{Qs}

We note that above two questions concern the following general problem of resolvability \cite{b8}.
For a group $G$ and a cardinal $\kk$, we denote 
$$res_\lambda(G,\kk)=\sup\{|\PP|:\PP\text{ is a partition of }G\text{ into left } \kk\text{-large subsets}\},$$
$$res(G,\kk)=\sup\{|\PP|:\PP\text{ is a partition of }G\text{ into left and right } \kk\text{-large subsets}\}.$$

{\em Given a group $G$ and a cardinal $\kk$, detect or evaluate the cardinals $res_\lambda(G,\kk)$ and $res(G,\kk)$.}

$5.$ By \cite[Corollary 2.7]{b5}, every infinite group $G$ admits a maximal regular left invariant topology (see also \cite[Chapter 9]{b1}).
\begin{Qs} Let $G$ be an Abelian group of singular cardinality $\kk$.
Does $G$ admit a maximal regular translation invariant $\kk$-bounded topology? \end{Qs}

\end{document}